\pdfoutput=1

\documentclass{IEEEtran}

\usepackage[utf8]{inputenc}
\usepackage{amsfonts}
\usepackage{amsmath}
\usepackage{amssymb}
\usepackage{mathtools}
\usepackage{amsthm}

\usepackage{todonotes}
\mathtoolsset{showonlyrefs}

\newcommand{\sign}[1]{\mbox{sign}(#1)}
\newcommand{\sgn}[1]{\lfloor#1\rceil}

\newtheorem{theorem}{Theorem}
\newtheorem{lemma}{Lemma}
\newtheorem{corollary}{Corollary}
\newtheorem{proposition}{Proposition}

\newtheorem{assumption}{Assumption}
\newtheorem{definition}{Definition}
\newtheorem{example}{Example}

\begin{document}

\title{
A methodology for designing fixed-time systems with a predefined upper-bound in their settling time
}

\author{Rodrigo~Aldana-L\'opez*, David~G\'omez-Guti\'errez*, Marco~Tulio~Angulo and  Michael Defoort
\thanks{*Equally contributing authors. Corresponding Author: D.~G\'omez-Guti\'errez, david.gomez.g@ieee.org}
\thanks{R.~Aldana-L\'opez is with the Department of Computer Science and Systems Engineering, University of Zaragoza, Zaragoza, Spain. D.~G\'omez-Guti\'errez is with the Multi-Agent Autonomous Systems Lab, Intel Labs, Intel Tecnología de México, Jalisco, Mexico and with Tecnologico de Monterrey, Escuela de Ingenier\'ia y Ciencias, Jalisco, Mexico. M.~T.~Angulo is with CONACyT - Institute of Mathematics, Universidad Nacional Autonoma de México, Juriquilla, Mexico. M. Defoort is with LAMIH, UMR CNRS 8201, Polytechnic University of Hauts-de-France, Valenciennes, France.}
}







\maketitle

\begin{abstract}
Algorithms having uniform convergence with respect to their initial condition (i.e., with fixed-time stability) are receiving increasing attention for solving control and observer design problems under time constraints. However, we still lack a general methodology to design these algorithms for high-order perturbed systems when we additionally need to impose a user-defined upper-bound on their settling time, especially for systems with perturbations.  Here, we fill this gap by introducing a methodology to redesign a class of asymptotically, finite- and fixed-time stable systems into non-autonomous fixed-time stable systems with a user-defined upper-bound on their settling time. Our methodology redesigns a system by adding time-varying gains. However, contrary to existing methods where the time-varying gains tend to infinity as the origin is reached, we provide sufficient conditions to maintain bounded gains.  We illustrate our methodology by building fixed-time online differentiators with user-defined upper-bound on their settling time and bounded gains.
\end{abstract}


\begin{IEEEkeywords}
Fixed-time stability, prescribed-time, nonlinear observers, online differentiators
\end{IEEEkeywords}

\section{Introduction}

There is an increasing attention in the study of systems having uniform stability with respect to their initial conditions (i.e., with fixed-time stability) for their potential applications in designing algorithms that satisfy time constraints~\cite{Andrieu2008,Polyakov2012,Polyakov2016,Aldana-Lopez2018,Sanchez-Torres2018,Holloway2019,Song2018}. Of particular interest are algorithms that additionally allow the user to set the desired Upper-Bound on their Settling Time (\textit{UBST}). However, the three classes of existing methodologies for designing fixed-time systems are limited.

First, methodologies based on homogeneity often lead to an \textit{UBST} that is unknown or significantly overestimated~\cite{Angulo2013,Basin2016b,Basin2016,Menard2017,Zimenko2018}. Second, methodologies based on a Lyapunov analysis provide the least \textit{UBST} for scalar systems, see e.g.~\cite{Aldana-Lopez2018,aldana2019design}. However, their extension to high-order systems is always very challenging due to the difficulty of finding adequate Lyapunov functions~\cite{Polyakov2012,Aldana-Lopez2018,Sanchez-Torres2018,aldana2019design}. Finally, the third class of methodologies obtain fixed-time system by incorporating time-varying gains, achieving convergence at exactly the desired time~\cite{Song2017,Song2018,Holloway2019}. However, applying this methodology requires that the gains become infinite at the convergence time ~\cite{Song2018} or produce Zeno behavior~\cite{Liu2018}. Both requirements significantly limit the applicability of this methodology.

In this Note, we introduce a methodology based on time-varying gains to redesign a class of perturbed asymptotic, finite-time, or fixed-time stable systems into fixed-time stable systems with an user-defined \textit{UBST}. Crucially, we provide sufficient conditions such that our methodology yields bounded gains. Moreover, contrary to other methodologies such as~\cite{Perruquetti2008,Holloway2019}, our methodology allows for exogenous disturbances.
By allowing disturbances, our methodology allows us to build single-input single-output unknown input observers and online differentiators with fixed-time stability, bounded gains, and desired UBST.
In particular, we apply our methodology to redesign the fixed-time sliding mode differentiator given of~\cite{Cruz-Zavala2011}, whose estimated \textit{UBST} is very conservative. The resulting redesigned non-autonomous fixed-time differentiator has an \textit{UBST} that is significantly less conservative and which can be easily chosen by the user.

The rest of this Note is organized as follows. In Section~\ref{Sec:ProbStatement}, we introduce the problem statement and our main result, providing a methodology based on time-varying gains to obtain a non-autonomous fixed-time stable system with desired \textit{UBST}. In Section~\ref{Sec:Design}, we illustrate our methodology by designing online differentiation algorithms with predefined-time convergence and desired UBST. We end in Section~\ref{Sec:Conclu} presenting some concluding remarks and discussing some limitations of our approach. 
Proofs of the main results are collected in the appendix.

\vspace{0.3cm}

\noindent \textbf{Notation:}
$\mathbb{R}$ is the set of real numbers, $\Bar{\mathbb{R}}=\mathbb{R}\cup\{-\infty,+\infty\}$, $\mathbb{R}_+=\{x\in\mathbb{R}\,:\,x\geq0\}$ and $\Bar{\mathbb{R}}_+=\mathbb{R}_+\cup\{+\infty\}$.
Given $\mathbf{r} = (r_1,\dots,r_n)^T\in\mathbb{R}^n$ and $\varrho = \sum_{i}^nr_i$, we denote the homogeneous norm by $\|x\|_\mathbf{r} = \left(x_1^{\varrho/r_1}+\dots+x_n^{\varrho/r_n}\right)^{1/\varrho}$. For $x\in\mathbb{R}$, $\sgn{x}^\alpha = |x|^\alpha \mbox{sign}(x)$, if $\alpha\neq0$ and $\sgn{x}^\alpha = \mbox{sign}(x)$ if $\alpha=0$. We use $I\in\mathbb{R}^{n\times n}$ for the identity matrix. $C^k(\mathcal{I})$ is the class of functions $f:\mathcal{I}\to\mathbb{R}$ with domain $\mathcal{I}\subseteq\mathbb{R}$  having continuous derivatives up to the $k$-th order on $\mathcal{I}$. $h'(z) = \frac{dh(z)}{dz}$ denotes the first derivative of function $h:\mathbb{R}\to\mathbb{R}$. For functions $\phi,\psi:\mathbb{R}\to\mathbb{R}$, $\phi\circ\psi(t)$ denotes the composition $\phi(\psi(t))$. For a function $\phi:\mathcal{I}\to\mathcal{J}$, its reciprocal $\phi(\tau)^{-1}$, $\tau\in\mathcal{I}$,  is such that $\phi(\tau)^{-1}\phi(\tau)=1$ and its inverse function $\phi^{-1}(t)$, $t\in\mathcal{J}$ is such that $\phi\circ\phi^{-1}(t)=t$.

\section{Problem statement and main results}\label{Sec:ProbStatement}

We start by introducing the notion of fixed-time stability and \textit{UBST}. Then, we introduce the class of systems that our methodology considers and the problem statement. Finally, we present our main result showing how to redesign this class of systems into fixed-time stable systems with desired \textit{UBST}. 

\subsection{Fixed-time stability.}

Consider the system
\begin{equation}\label{eq:sys}
    \dot{x}=f\left(x,t\right)+D\delta(t), \quad \forall t\geq t_0, 
\end{equation}
where $x\in\mathbb{R}^n$ is the state and $t\in[t_0,+\infty)$ is time. Above, $f:\mathbb{R}^n\times\mathbb{R}_+\to\mathbb{R}^n$ is some function that is continuous on $x$ (except, perhaps, at the origin), and continuous almost everywhere on $t$. $D=[0,\ldots,0,1]^T$ and $\delta(t)\in\mathbb{R}$ is a disturbance.
The solutions of~\eqref{eq:sys} are understood in the sense of Filippov~\cite{Cortes2008}. We assume that $f(\cdot,\cdot)$
is such that the origin of~\eqref{eq:sys} is asymptotically stable and, except at the origin, \eqref{eq:sys} has the properties of existence and uniqueness of solutions in forward-time on the interval $[t_0,+\infty)$~\cite[Proposition~5]{Cortes2008}. 
The set of admissible disturbances, on the interval $[t_0,\mathbf{t}]$ with $t_0<\mathbf{t}$ is denoted by $\mathcal{D}_{[t_0,\mathbf{t}]}$.
The solution of \eqref{eq:sys} for $t\in [t_0,\mathbf{t}]$, with disturbance $\delta_{[t_0,\mathbf{t}]}$ (i.e. the restriction of $\delta(t)$ to $[t_0,\mathbf{t}]$) and initial condition $x_0$ is denoted by $x(t;x_0,t_0,\delta_{[t_0,\mathbf{t}]})$, and the initial state is given by $x(t_0;x_0,t_0,\cdot) = x_0$. 

We assume that the origin is the unique equilibrium point of the systems under consideration.  Note that because $f(\cdot, \cdot)$ can be discontinuous at the origin, system \eqref{eq:sys} can have an equilibrium point at the origin despite the presence of disturbances.

For the system in Eq.~\eqref{eq:sys}, its \textit{settling-time function} $T(x_0, t_0)$ for the initial state $x_0 \in \mathbb R^n$ and the initial time $t_0 \geq 0$ is defined as:
\begin{multline}
T(x_0,t_0)=\inf\{\xi\geq t_0:\forall\delta_{[t_0,\infty)}\in\mathcal{D}_{[t_0,\infty)},\\ \lim_{t\to\xi}x(t;x_0,t_0,\delta_{[t_0,\infty)})=0\}-t_0.
\end{multline}
Without ambiguity, when $t_0=0$ we simply write $T(x_0)$.

For simplicity, in the rest of this Note we write ``stable" instead of ``the origin is globally stable". With this shorthand, we can introduce the notion of fixed-time stability.

\begin{definition}\cite{Polyakov2014} \label{def:fixed}
System \eqref{eq:sys} is  \textit{fixed-time stable} if it is asymptotically stable~\cite{Khalil2002} and the settling-time function $T(x_0,t_0)$ is bounded on  $\mathbb{R}^n\times\mathbb{R}_+$, i.e. there exists $T_{{\max}}<+\infty$ such that $$T(x_0,t_0)\leq T_{\max},$$ for all $t_0\in\mathbb{R}_+$ and $x_0\in\mathbb{R}^n$. The quantity $T_{\max}$ is called an Upper Bound of the Settling Time (\textit{UBST}) of the system~\eqref{eq:sys}.
\end{definition}

\subsection{Problem statement}

Our methodology considers the class of systems that can be written in the form
\begin{equation}\label{Eq:System1}
    \dot{x}=G(x_1)+A_0x+D\delta(t),
\end{equation}
Here, $x=[x_1,\ldots,x_n]^T$ is the state, $G:\mathbb{R} \rightarrow \mathbb{R}^n$ is some function, $D =[0,\cdots, 0, 1]^T \in \mathbb R^n$, and  $A_0=[a_{ij}]\in\mathbb{R}^{n\times n}$ is an upper-diagonal matrix with $a_{ij}=1$ if $j=i+1$ and $a_{ij}=0$ are constants. The function  $\delta(t)\in\mathbb{R}$ is a disturbance satisfying $|\delta(t)|\leq L$, for all $t\in[t_0,+\infty)$, with $L \geq 0$ a known constant.
This class of systems in Eq.~\eqref{Eq:System1} appears in the design of single-input single-output observers, single-input single-output unknown-input observers, and online differentiators~\cite{Perruquetti2008,Levant2013,Levant2019}. 

On system~\eqref{Eq:System1}, we make the following assumption:
\begin{assumption}
\label{Assump:Diff}
The function $G(\cdot)$ is such that system~\eqref{Eq:System1} is globally asymptotically stable.
\end{assumption}

The above assumption means that we already have a way to design an asymptotically stable system with the structure~\eqref{Eq:System1}. Of course, this system can also be finite-time stable.
There are several ways to satisfy Assumption~\ref{Assump:Diff}. For example, consider
\begin{equation}
G(x_1)=
\left[
\begin{array}{c}
l_1\sgn{x_1}^{\frac{n-m}{n}}\\
    \vdots\\
l_{n}\sgn{x_1}^{\frac{n-nm}{n}}
\end{array}
\right],
\end{equation}
where $m \geq 0$ and $\{l_i\}_{i=1}^n$ are parameters. For $L = 0$ (i.e., without disturbance) one can simply choose $m = 0$ and gains $\{l_i\}_{i=1}^n$ making the polynomial $s^n-l_1s^{n-1}-\cdots-l_n$ Hurwitz. For $L>0$, one can select $m=1$ and design $\{l_i\}_{i=1}^n$ as when using high-order sliding mode algorithms~\cite{Levant2003,Cruz2018}.

Given a desired \textit{UBST} $T_c>0$ and under Assumption~\ref{Assump:Diff}, our objective is to redesign~\eqref{Eq:System1} into another system
\begin{equation}\label{Eq:PrescDiff} 
\dot{y}=H(y_1,t,T_c)+A_0y+D\delta(t),
\end{equation}
$y = [y_1,\ldots, y_n]^T \in \mathbb{R}^n$, that is fixed-time stable with $T_c$ as an \textit{UBST}. Solving this redesign problem consists in designing a function $H(\cdot, \cdot,\cdot)$ that satisfies the above objective.

\subsection{Main results.}

Given $T_c$, our main result provides a methodology to design function $H(\cdot, \cdot,\cdot)$ starting from $G(\cdot)$ and using two additional functions $\{\rho,F\}$. Here,  $\rho:\mathbb{R}_+\to\mathbb{R}_+$ and $F:\mathbb{R}\to\mathbb{R}^n$.  
The first function is defined as
$$ \rho(\tau) := \frac{1}{T_c}\Phi(\tau)^{-1},$$ where the function $\Phi:\mathbb{R}_+\to\Bar{\mathbb{R}}_+\setminus\{0\}$ satisfies the following assumption:
\begin{assumption}
\label{Assump:NonAut}
The function $\Phi(\cdot)$ satisfies:
\begin{itemize}
    \item $\int_0^{+\infty} \Phi(z)dz = 1$,
    \item $\Phi(\tau)<+\infty$, for all $\tau>0$, 
    \item it is either non-increasing or locally Lipschitz on $\mathbb{R}_+\setminus\{0\}$.
\end{itemize}
\end{assumption}

A simple choice for $\Phi(\cdot)$ would be any probability density function. Note that Assumption \ref{Assump:NonAut} implies that $\lim_{z\to\infty}\Phi(z) = 0$.

The second function $F(\cdot)$ should be chosen to ensure that an ``auxiliary system" 
is asymptotically stable. More precisely, define the auxiliary system
\begin{multline}
\label{Eq:TauSystem}
    \frac{dz}{d\tau}=rF(z_1)+rA_0z-\rho(\tau)^{-1}\frac{d\rho(\tau)}{d\tau}Mz\\+(r\rho(\tau))^{-n}D\hat{\delta}(\tau),
\end{multline}
with initial time $\tau_0=0$ and initial condition $z(\tau_0;z_0,\tau_0,\cdot)= z_0$. Above, $z=[z_1,\ldots, z_n]^T \in \mathbb R^n$, $\tau$ is a ``new" time variable, $r >0$ is a parameter, $M=\mbox{diag}(0,1,\ldots,n-1)$, and $\hat{\delta}(\tau)$ is a new ``disturbance" also satisfying $|\hat{\delta}(\tau)|\leq L$ (the relation between $\delta(t)$ and $\hat{\delta}(\tau)$ will become apparent later). Importantly, note this new disturbance is multiplied by the $\rho(\tau)^{-n}$. Thus, their product vanishes as $\tau \rightarrow \infty$.
The function $F(\cdot)$ should satisfy the following assumption:
\begin{assumption}
\label{Assump:Main}
The functions $\{\rho, F\}$ are chosen such that:
\begin{itemize}
\item[(i)] The origin of the auxiliary system~\eqref{Eq:TauSystem} is asymptotically stable with a settling time function $\mathcal T(z_0)$ satisfying $\mathcal T(z_0) \leq T_{f}\in\Bar{\mathbb{R}}_+$ for a known constant $T_{f}$.
\item[(ii)] The following conditions are satisfied for any disturbance $\hat{\delta}(\tau)$: 
\begin{equation}
\label{Eq:CondTrans}
\lim_{\tau\to\infty}\rho(\tau)^{i-1}z_i(\tau;z_0,0,\hat{\delta}_{[0,\infty)})=0, \quad i = 1, \cdots, n.
 \end{equation}
Here $z_i(\tau;z_0,0,\hat{\delta}_{[0,\infty)})$ is the $i$-th variable of the solution of~\eqref{Eq:TauSystem}.
\end{itemize} 
\end{assumption}

Note that Assumption~\ref{Assump:Main}-(i) does not require that system~\eqref{Eq:TauSystem} is fixed-time stable because it allows $T_{f}=+\infty$. We can also get $T_{f}=+\infty$ when ~\eqref{Eq:TauSystem} is fixed-time stable but there is no known \textit{UBST}. When system~\eqref{Eq:TauSystem} is fixed-time stable with $T_{\max}^*$ as a known \textit{UBST}, we write $T_f=T_{\max}^*$.

In Section~\ref{Subsec:Design} we illustrate how to choose $\{\rho,F\}$ to satisfy Assumptions~\ref{Assump:NonAut} and~\ref{Assump:Main}. In particular, it is shown that $F(\cdot)$ can be linear, or nonlinear with nonlinearities taken from the High-Order Sliding Mode (\textit{HOSM}) differentiator~\cite{Levant2003}.  Note that Assumption~\ref{Assump:Main}-(ii) is immediately satisfied for any function $\rho(\cdot)$ when the auxiliary system is finite- or fixed-time stable (e.g., as in the \textit{HOSM} differentiator). An analogous result can be derived when $F(\cdot)$ is linear (see Lemma~\ref{lemma:exp_stability} in Appendix~\ref{AppendixAuxResults}).

To introduce our main result,  let
\begin{equation}
\label{Eq:NonAutpsi}
\psi(\tau) = T_c \int_0^\tau \Phi(\xi) d \xi.
\end{equation}
Define 
\begin{equation}
\label{Eq:Eta}
\eta:=\lim_{\tau\to T_{f}} T_c^{-1}\psi(\tau)
\end{equation}
and 
$$\kappa(t -t_0) :=
\left\lbrace
\begin{array}{ccc}
    \rho \circ \psi^{-1}(t - t_0) & \text{for} &  t\in[t_0,\eta T_c) \\
    1 & & \text{otherwise.} 
\end{array}
\right.$$ 
Our main result is the following:

\begin{theorem} 
\label{Th:MainResult}
Suppose that:
\begin{itemize}
\item[(i)] The original system~\eqref{Eq:System1} satisfies Assumption~\ref{Assump:Diff} (i.e., it is globally asympotically stable)
\item[(ii)] The functions $\{\rho,F\}$ of the auxiliary system~\eqref{Eq:TauSystem} are chosen to satisfy Assumptions~\ref{Assump:NonAut} and~\ref{Assump:Main}.
\end{itemize}
Then, choosing $H(\cdot,\cdot,\cdot)$ as 
\begin{equation}
    H(y_1,t,T_c)=\left\lbrace
    \begin{array}{cl}
       \Lambda(r\kappa(t-t_0))F(y_1)  & \text{for }  t\in[t_0,t_0+\eta T_c),  \\
       G(y_1)  &  \text{otherwise,}
    \end{array}
    \right.
\end{equation}
where $\Lambda(k)=\mbox{diag}(k,k^2,\ldots,k^n)$, makes the redesigned system~\eqref{Eq:PrescDiff} fixed-time stable with \textit{UBST} given by $\eta T_c$.
\end{theorem}
\begin{proof} See Appendix~\ref{AppendProofMain}. \end{proof}

In the above theorem, the redesign of $G(\cdot)$ that we propose consists in a $H(\cdot, \cdot,\cdot)$ that is piecewise. First, the system starts using function $F(\cdot)$ multiplied by the time-varying gain $\Lambda(r\kappa(t-t_0))$. This choice drives the state of the system to the origin in a time upper bounded by $\eta T_c$. Then, after $\eta T_c$ time units, the function $H(\cdot, \cdot,\cdot)$ simply takes its original form $G(\cdot)$ to maintain the state of the system at the origin despite the disturbance. 

In Theorem~\ref{Th:MainResult}, because the auxiliary system is asymptotically stable with settling time function $\mathcal T(z_0)$, we prove in Appendix~\ref{AppendProofMain} that the settling time function of the redesigned system~\eqref{Eq:PrescDiff} is
$$T(y_0,t_0)=\lim_{\tau\to\mathcal{T}(z_0)}\psi(\tau)=T_c\int_{0}^{\mathcal{T}(z_0)}\Phi(\xi)d\xi.$$
From this expression, we can immediately conclude that $T(y_0, t_0) \leq \eta T_c$ for all $y_0$ and $t_0$ because by construction, $\Phi$ satisfies $\int_0^\infty \Phi(\xi) d\xi =1$. Furthermore, one can obtain a more detailed description of the convergence properties that the redesigned system inherits from the auxiliary system.

\begin{proposition} 
\label{Prop:Main}
Under the conditions of Theorem~\ref{Th:MainResult}, the following holds:
    \begin{enumerate}
    \item if $\mathcal{T}(z_0)=\infty$, for all $z_0\in\mathbb{R}^n\setminus \{0\}$, then the settling time of every nonzero trajectory of the redesigned system~\eqref{Eq:PrescDiff} is precisely $T_c$.
    \item if $\mathcal{T}(z_0)$ is finite but radially unbounded, then $T_c$ is the least \textit{UBST} of~\eqref{Eq:PrescDiff}.
    \item Suppose that there exists $T_{f}<+\infty$ such that for all $z_0 \in \mathbb{R}^n$, $\mathcal{T}(z_0)\leq T_{f}$, (i.e., ~\eqref{Eq:TauSystem} is fixed-time stable). Then, there exists $\hat{T}_c<T_c$ such that $\hat{T}_c$ is an \textit{UBST} function of~\eqref{Eq:PrescDiff}, i.e. for all $y_0\in\mathbb{R}^n\setminus\{0\}$, $T(y_0)<\lim_{\tau\to T_{f}}\psi(\tau)<\hat{T}_c<T_c$. 
    \end{enumerate}
\end{proposition}
\begin{proof} See Appendix~\ref{AppendixPropProofs}.
\end{proof}

The above observations are relevant because they allow us to obtain the following two results. First, Corollary~\ref{Cor:TightBound} states that, even if the \textit{UBST} in the auxiliary system~\eqref{Eq:TauSystem} is very conservative, by adequately choosing the function $\rho(\tau)$, our methodology actually yields an \textit{UBST} that is arbitrarily tight. Then, Corollary~\ref{Cor:BoundedGain}, allows characterizing conditions under which our methodology provides bounded gains.

\begin{corollary}
\label{Cor:TightBound}
Let $s_\rho=\eta T_c-\sup_{(x_0,t_0) \in \mathbb{R}^n\times\mathbb{R}_+}T(x_0,t_0)$ be the slack between the least \textit{UBST} and the predefined one given by $\eta T_c$. Then, for any $\varepsilon\in\mathbb{R}_+\setminus\{0\}$ there exists a possible choice of $\rho(\tau)$ such that    $
    s_\rho\leq \varepsilon$.   
\end{corollary}

To obtain the above result, note that if $T_f^*=\sup_{z_0\in \mathbb{R}^n}\mathcal{T}(z_0)$ one can take $\rho(\tau) = \frac{1}{\alpha T_c}\exp(\alpha\tau)$ with $\alpha>0$. Thus, $\eta=1-\exp(-\alpha T_{\max}^*)$ and $\psi(\tau)=T_c(1-\exp(-\alpha \tau))$. Therefore, $\alpha$ exists such that $s_\rho=T_c(\exp(-\alpha T_{f}^*)-\exp(-\alpha T_{\max}^*))\leq\epsilon$. 

\begin{corollary} 
\label{Cor:BoundedGain}
Assume that there exists a known constant $T_{\max}^*<+\infty$ such that $\sup_{z_0 \in \mathbb{R}^n}\mathcal{T}(z_0)\leq T_{\max}^*$. Then, the time-varying gain $\kappa(t-t_0)$ of $H(\cdot,\cdot,\cdot)$ of the redesigned system~\eqref{Eq:PrescDiff}, is bounded as
$$
\kappa(t -t_0) \leq\lim_{t\to \eta T_c}\rho \circ \psi^{-1}(t - t_0)=\lim_{\tau\to T_{\max}^*} \rho(\tau)
<+\infty.$$

\end{corollary}

The above result shows that if  $\{\rho, F\}$ are chosen to ensure that the auxiliary system~\eqref{Eq:TauSystem} is fixed-time stable with a known \textit{UBST}, then the redesigned system can be fixed-time stable with a predefined \textit{UBST} and with bounded gains.

\subsection{Examples: designing fixed-time stable systems with desired \textit{UBST}.}
\label{Subsec:Design}

Here, we apply our methodology illustrating three possible choices for the functions $\{F(\cdot), \rho(\cdot)\}$.
We start with the simplest case when the function $F(\cdot)$ is linear. This case is associated to item 1 of Proposition~\ref{Prop:Main}.
\begin{proposition}
\label{Prop:Lineal}
Let 
\begin{equation}
    F(z_1)=\left[\begin{array}{c}
         k_1 z_1  \\
         \vdots \\
         k_n z_1
    \end{array} 
    \right] ,
\end{equation} 
and choose $K=[k_1\ \cdots\ k_n]^T$ such that $A=A_0+KC$ is Hurwitz, with $C=[1\ 0\ \cdots\ 0]$. Then, the redesigned system~\eqref{Eq:PrescDiff} is fixed-time stable with $T_c$ as an \textit{UBST} if  $\rho: \mathbb{R}_+ \rightarrow \mathbb{R}_+$ and $r$ satisfy at least one of the following conditions:
\begin{enumerate}
    \item $\rho(\tau)^{-1}\frac{d\rho(\tau)}{d\tau}=1$, and $r>2\lambda_{\max}(P)(n-1)$ where $P$ is the solution of $PA+A^TP=-I$,
    \item $\rho(\tau)^{-1}\frac{d\rho(\tau)}{d\tau}\to 0$ as $\tau\to\infty$, and $r>0$.
\end{enumerate}
\end{proposition}
\begin{proof}
See Appendix~\ref{AppendixPropProofs}.
\end{proof}

Examples of $\rho(\cdot)$ satisfying Condition~1 and Condition~2 are $\rho(\tau) = \frac{1}{T_c}\exp(\tau)$ and $\rho(\tau) = \frac{\pi}{2T_c}(\tau^2 + 1)$, respectively.

In the particular case when $L = 0$ (i.e., $\delta(t)\equiv0$), Proposition~\ref{Prop:Lineal} is similar to the main result in~\cite{Holloway2019}. However, our time-varying gains are simpler and allow using different classes of time-varying gains $\kappa(t-t_0)$, other than the time-base generators proposed in~\cite{Morasso1997}. Moreover, unlike~\cite{Holloway2019} that was limited to the case with $L=0$, our result with $L>0$ allows the application to unknown input observers and exact differentiators with predefined-time convergence (i.e. with $\delta(t)\not\equiv0$).

Note that, since $F(\cdot)$ is linear, system~\eqref{Eq:TauSystem} is asymptotically stable, $T_f=+\infty$ and $\eta=1$. Thus, similar as in~\cite{Holloway2019}, the time-varying gain is unbounded: $\lim_{t\to t_0+T_c}\kappa(t-t_0)=+\infty$. 

Next, we show that our methodology yields a system~\eqref{Eq:PrescDiff} that is fixed-time stable with least \textit{UBST} given by $T_c$. For this aim, consider the case when $F(\cdot)$ is such that system~\eqref{Eq:TauSystem} is finite-time stable. This case is associated to item 2 of Proposition~\ref{Prop:Main}. In this case, for any finite initial condition, the following result guarantees that the origin is reached before the singularity in $\kappa(t-t_0)$ appears.

\begin{proposition}
\label{Prop:STFixed}
Let 
\begin{equation}
    F(z_1)=\left[\begin{array}{c}
         l_1 \sgn{z_1}^{\frac{n-1}{n}}  \\
         l_2 \sgn{z_1}^{\frac{n-2}{n}} \\
         \vdots \\
         l_n \sgn{z_1}^{0}
    \end{array} 
    \right].
\end{equation} 
with $\{l_i\}_{i=1}^n$ chosen as in~\cite{Levant2003,Levant2019}. Then, the redesigned system~\eqref{Eq:PrescDiff} is fixed-time stable with $T_c$ as its least \textit{UBST} if $r = 1$ and the function $\rho$ is $C^n([0,\infty))$ and satisfies $$\rho(\tau)^{-1}\frac{d\rho(\tau)}{d\tau}\to 0 \mbox{ as } \tau\to\infty.$$
\end{proposition}
\begin{proof}
See Appendix~\ref{AppendixPropProofs}.
\end{proof}

Proposition~\ref{Prop:STFixed} allows us to design fixed-time exact differentiators based on Levant's high-order sliding mode differentiator~\cite{Levant2003}. Here, one simple choice is again $\rho(\tau) = \frac{\pi}{2T_c}(\tau^2 + 1)$.
In~\cite{Levant2014}, an arbitrary order differentiator using time-varying gains was proposed. Compared to~\cite{Levant2014}, our approach guarantees a fixed-time convergence with predefined \textit{UBST}.

Finally, we consider the case when the auxiliary system takes the form of~\eqref{Eq:System1}. To this end, let us consider the following linear system:
\begin{equation}
\label{eq_syst_last}
    \frac{dz}{d\tau}=\Gamma_z z, \quad y_z = C z,
\end{equation}
where $y_z$ is the output, and
$\Gamma_z=[\gamma_{ij}] \in \mathbb R^{n \times n}$ is such that
$$\gamma_{ij}=\left\lbrace 
\begin{array}{cl}
1 & \text{if }  i=j+1  \\
-\alpha(i-1) & \text{if } i=j\\ 
0 & \text{otherwise}
\end{array}
\right.$$
$\alpha>0$ and $C=\left[
    \begin{array}{ccccc}
        1 & 0 & \cdots & 0 & 0
    \end{array}
    \right]$.
The characteristic polynomial of $\Gamma_z$ is $s^n+a_1s^{n-1}+\cdots+a_n$. The similarity transformation $\mathcal{Q}\in \mathbb R^{n \times n}$ (i.e. $\hat{z}=\mathcal{Q}z$) that transforms system \eqref{eq_syst_last} into its observer canonical form~\cite{Kailath1980} is denoted by:
\begin{equation}
\mathcal{Q}:=(\mathcal{V}\mathcal{O}(\Gamma,C))^{-1}.
\label{Eq:SimTrans}
\end{equation}
Here
$\mathcal{O}(\Gamma,C)$ is the observability matrix of the pair $(\Gamma,C)$ and $\mathcal{V}:=[v_{ij}]$ is given by
$$
v_{ij}=\left\lbrace
\begin{array}{cl}
    1 &  \text{if }  j=i \\
    a_{i-j} & \text{if } j<i\\
    0     &  \text{otherwise}.
\end{array}
\right.
$$

The observer canonical form is given by the following dynamics
\begin{align}
    \frac{d\hat{z}_i}{d\tau}&=-a_i\hat{z}_1+\hat{z}_{i+1} \quad i=1,\ldots,n-1,\\
    \frac{d\hat{z}_n}{d\tau}&=-a_n\hat{z}_1.
\end{align}

Notice that $\mathcal{Q}^{-1}D=D$.

\begin{proposition}
\label{Basin}
Choose $\rho(\tau)$ such that $\rho(\tau)^{-1}\frac{d\rho(\tau)}{d\tau}=\alpha>0$ and set $r=1$. Let $k_i\in \mathbb{R}$ and $g_i:\mathbb{R} \to \mathbb{R}$, $i=1,\ldots,n$, be such that
\begin{align*}
\frac{dz_i}{d\tau}&=k_ig_i(z_1)+z_i ,\quad i=1, \cdots, n-1,\\
\frac{dz_n}{d\tau}&=k_ng_n(z_1)+\hat{\delta}(\tau),
\end{align*}
is asymptotically stable with settling time function $\mathcal{T}(z_0)$ satisfying $\mathcal{T}(z_0)\leq T_{f}$ for all $z_0\in\mathbb{R}^n$. 
Choose
$$
F(x_1)=Q
\left[
\begin{array}{c}
     k_1g_1(x_1)+a_1x_1  \\
     \vdots\\
     k_ng_n(x_1)+a_nx_1
\end{array}
\right].
$$with $Q \in \mathbb R^{n \times n}$  the similarity transformation given by~\eqref{Eq:SimTrans}. 
Then, under the conditions of Theorem 1, the redesigned system~\eqref{Eq:PrescDiff} if fixed-time stable with $\eta T_c$ as the predefined \textit{UBST}. Moreover, if $T_f=T_{\max}^*< +\infty$, the gains $\kappa(t-t_0)$ remain bounded. 
\end{proposition}
\begin{proof}
See Appendix~\ref{AppendixPropProofs}.
\end{proof}

The above result allows us to obtain from an autonomous fixed-time stable system, a non-autonomous fixed-time stable system with predefined \textit{UBST}. Furthermore, it guarantees that the time-varying gains remain bounded.
If the functions $g_i(\cdot)$ are such that~\eqref{Eq:System1} is fixed-time stable, such as in~\cite{Basin2016,Menard2017,Lopez2018finite,Angulo2013}, then the above result is concerned with Proposition~\ref{Prop:Main}-item-3). If in addition, an \textit{UBST} is known, as in~\cite{Cruz-Zavala2011} the above result is concerned with Corollary~\ref{Cor:BoundedGain}.

Compared to the results derived for autonomous systems~\cite{Basin2016,Menard2017,Lopez2018finite,Cruz-Zavala2011,Angulo2013}, our approach allows to tune the parameters such that the overestimation of the \textit{UBST} is significantly reduced as it will be illustrated in Example~\ref{ex:cruz_zavala}. 

\section{Application to the design of differentiators with fixed-time stability }
\label{Sec:Design}

Here we illustrate the application of our methodology to obtain fixed-time online differentiators with predefined \textit{UBST} from the family of \textit{HOSM} differentiators. More precisely, the following corollary allows deriving high-order fixed time differentiators with time-varying gains from our previous results. Such differentiators are designed such that the differentiation error coincides with~\eqref{Eq:PrescDiff}. 

We illustrate this design to have the form of the Levant's filtering differentiator~\cite{Levant2019}, where $n_f$ is the filter order and $n_d$ is the number of derivatives to be obtained. Thus, after $\eta T_c$, the filtering properties coincide with~\cite{Levant2019}. 

\begin{corollary}
\label{Cor:BasinDiff}
Let $r$, $\rho(\tau)$, and $F(z_1)=[f_1(z_1),\ldots,f_n(z_1)]^T$ be such that the condition of Theorem~\ref{Th:MainResult} are satisfied with $T_{f}\in\Bar{\mathbb{R}}_+$ be such that $\mathcal{T}(z_0)\leq T_{f}$. 
Moreover, let $y(t) \in \mathbb{R}$ be a continuous function $n_d+1$ times differentiable, such that $|\frac{d^{{n_d}+1}}{dt^{n_d+1}}y(t)|\leq L$ and consider the algorithm
\begin{align}
&\dot{w}_i=\left\lbrace 
\begin{array}{lll}
-r^i\kappa(t-t_0)^ik_if_i(w_1)+w_{i+1}     & t\in[t_0,t_0+\eta T_c)\\
-l_i\sgn{w_1}^{\frac{n_d+n_f+1-i}{n_d+n_f+1}}+w_{i+1}     & \text{otherwise}
\end{array}
\right.
\intertext{for $i=1,\ldots,n_f$,}
&\dot{z}_0=\left\lbrace 
\begin{array}{lll}
-r^i\kappa(t-t_0)^ik_if_i(w_1)+z_{1}-y(t)     & t\in[t_0,t_0+\eta T_c)\\
-l_i\sgn{w_1}^{\frac{n_d+n_f+1-i}{n_d+n_f+1}}+z_{1}-y(t)    & \text{otherwise}
\end{array}
\right.
\intertext{for $i=n_f+1$, and}
&\dot{z}_{i-n_f-1}=\left\lbrace 
\begin{array}{lll}
-r^i\kappa(t-t_0)^ik_if_i(w_1)+z_{i-n_f}     & t\in[t_0,t_0+\eta T_c)\\
-l_i\sgn{w_1}^{\frac{n_d+n_f+1-i}{n_d+n_f+1}}+z_{i-n_f}     & \text{otherwise}
\end{array}
\right.
\end{align}
for $i=n_f+2,\ldots,n_d+n_f+1$, where $z_{n_d+1}=0$; $\kappa(t-t_0)=\rho(\psi^{-1}(t-t_0))$, $\psi$ as in~\eqref{Eq:NonAutpsi} and $\eta=\lim_{\tau\to T_{f}}\frac{1}{T_c}\psi(\tau)$. Then, for all $t>t_0+\eta T_c$ one has, $w_i=0$ for all $i=1,\ldots,n_f$ and $z_i=\frac{d^i}{dt^i}y(t)$ for all $i=0,\ldots,n_d$. For $n_f=0$, $w_1$ is defined as $w_1=z_0-y(t)$.
\end{corollary}
The result of Corollary~\ref{Cor:BasinDiff} follows trivially by noticing that taking $x_i=w_i$, for $i=1,\ldots,n_f$ and $x_{n_f+1+i}=z_i-\frac{d^i}{dt^i}y(t)$ leads to system~\eqref{Eq:PrescDiff}, where $n=n_d+n_f+1$ and $\delta(t)=\frac{d^{n_d+1}}{dt^{n_d+1}}y(t)$.

\begin{example}
\label{ex:raw_levant}
Consider system~\eqref{Eq:PrescDiff} with $F(x_1)$ as in Proposition \ref{Prop:STFixed} with $n=3$, $l_1=2L^\frac{1}{3}$, $l_2=2.12L^\frac{2}{3}$, $l_3=1.1L$ and $L=2.2$. With these parameters, let us apply directly Corollary \ref{Cor:BasinDiff} with $n_d=1$.
In this case $\rho(\tau) = \frac{\pi}{2T_c}(\tau^2 + 1)$ which verifies $\lim_{\tau\to\infty}\rho(\tau)^{-1}\frac{d\rho(\tau)}{d\tau}= 0$, $\int_0^\infty (T_c\rho(\tau))^{-1}d\tau = 1$, and $\kappa(t-t_0) = \frac{\pi}{2}\sec^2\left(\frac{\pi(t-t_0)}{2T_c}\right)$. The resulting system was simulated in order to differentiate $y(t)=-0.4\sin(t)+0.8\cos(0.8t)$. Figure \ref{fig:raw_levant} shows the trajectories of $x(t)=(x_1(t),x_2(t),x_2(t))^T$ with $x(0) = (100,0,0)^T$ as well as the values of $y(t)$ and $\dot{y}(t)$. Moreover, the performance of the system with measurement noise $0.01\cos(10t)+0.001\cos(30t)$ is shown.

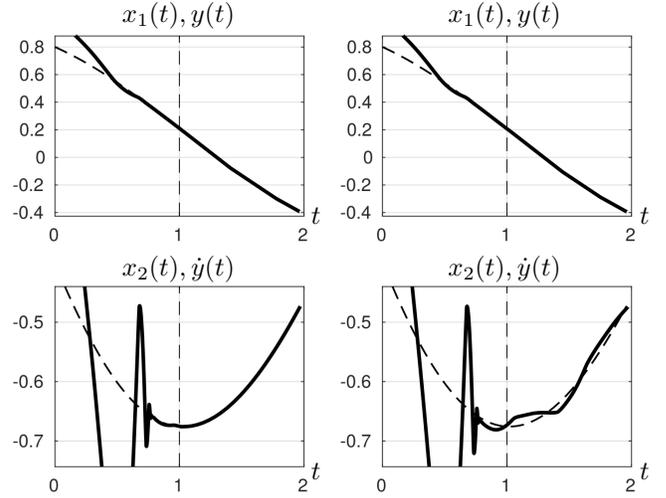
\begin{figure}
    \centering
\def\svgwidth{8.5cm}    
\begingroup%
  \makeatletter%
  \providecommand\color[2][]{%
    \errmessage{(Inkscape) Color is used for the text in Inkscape, but the package 'color.sty' is not loaded}%
    \renewcommand\color[2][]{}%
  }%
  \providecommand\transparent[1]{%
    \errmessage{(Inkscape) Transparency is used (non-zero) for the text in Inkscape, but the package 'transparent.sty' is not loaded}%
    \renewcommand\transparent[1]{}%
  }%
  \providecommand\rotatebox[2]{#2}%
  \newcommand*\fsize{\dimexpr\f@size pt\relax}%
  \newcommand*\lineheight[1]{\fontsize{\fsize}{#1\fsize}\selectfont}%
  \ifx\svgwidth\undefined%
    \setlength{\unitlength}{352.15393066bp}%
    \ifx\svgscale\undefined%
      \relax%
    \else%
      \setlength{\unitlength}{\unitlength * \real{\svgscale}}%
    \fi%
  \else%
    \setlength{\unitlength}{\svgwidth}%
  \fi%
  \global\let\svgwidth\undefined%
  \global\let\svgscale\undefined%
  \makeatother%
  \begin{picture}(1,0.80971289)%
    \lineheight{1}%
    \setlength\tabcolsep{0pt}%
    \put(0,0){\includegraphics[width=\unitlength]{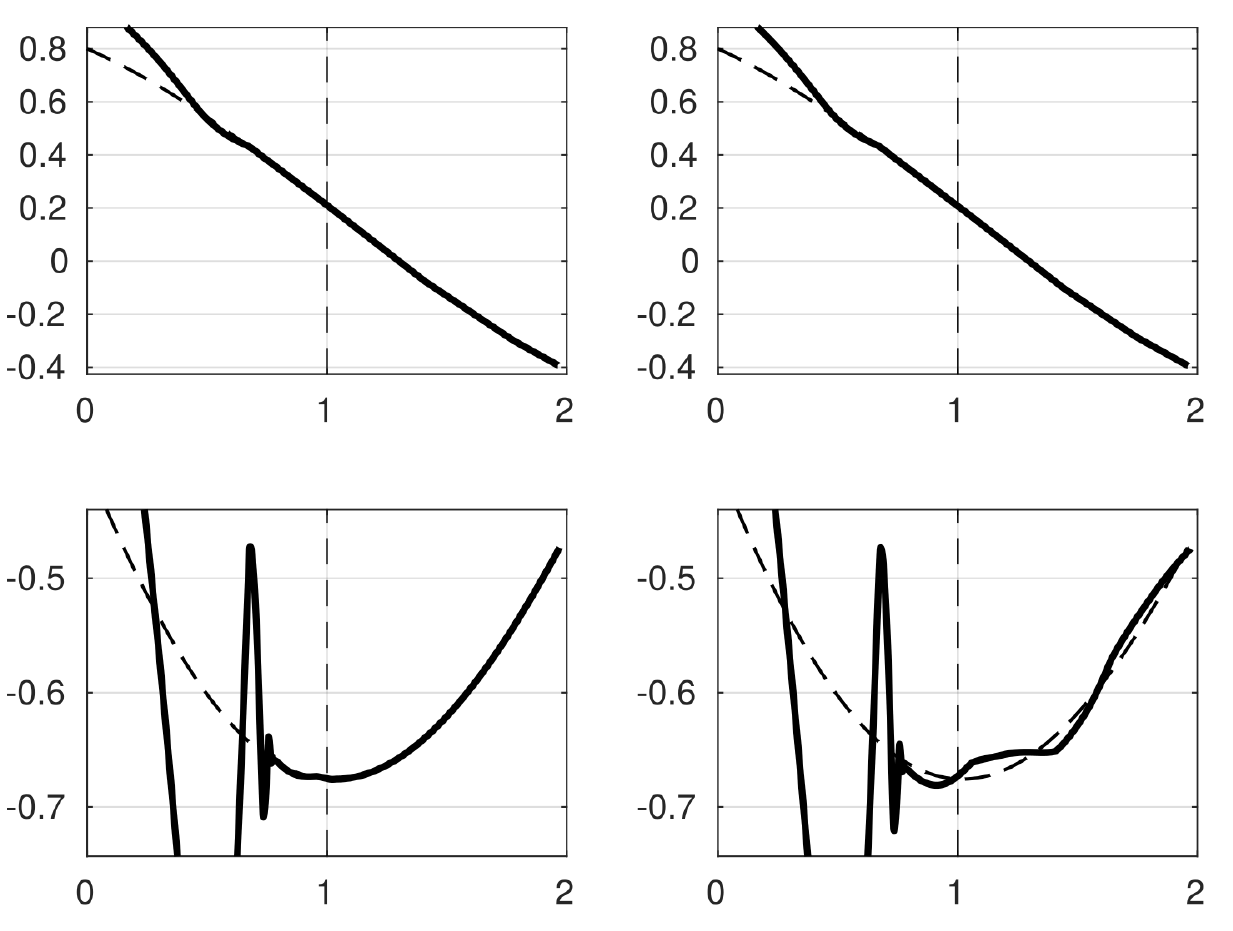}}%
    \put( 0.160482251,0.744956467){  \rotatebox{0}{\makebox(0,0)[lb]{$x_1(t),y(t)$}}
    }%
     \put( 0.670482251,0.744956467){  \rotatebox{0}{\makebox(0,0)[lb]{$x_1(t),y(t)$}}
    }%
    \put( 0.160482251,0.34956467){  \rotatebox{0}{\makebox(0,0)[lb]{$x_2(t),\dot{y}(t)$}}
    }%
     \put( 0.670482251,0.34956467){  \rotatebox{0}{\makebox(0,0)[lb]{$x_2(t),\dot{y}(t)$}}
    }%
    
    \put( 0.470482251,0.43956467){\rotatebox{0}{\makebox(0,0)[lb]{$t$}}}%
    \put( 0.98482251,0.43956467){\rotatebox{0}{\makebox(0,0)[lb]{$t$}}}%
    \put( 0.470482251,0.03956467){\rotatebox{0}{\makebox(0,0)[lb]{$t$}}}%
    \put( 0.98482251,0.03956467){\rotatebox{0}{\makebox(0,0)[lb]{$t$}}}%
    
  \end{picture}%
\endgroup%
    \caption{Online differentiation of the signal $y(t)=-0.4\sin(t)+0.8\cos(0.8t)$ with predefined-time convergence at $T_c=1$ of Example~\ref{ex:raw_levant}, in the noiseless case (left) and with noise (right).}
    \label{fig:raw_levant}
\end{figure}

\end{example}

\begin{example}
\label{ex:marco_tulio}
According to \cite{Angulo2013}, the system
\begin{equation}
\begin{aligned}
    \frac{dz_1}{d\tau} &= -2L^\frac{1}{3}\theta\sgn{z_1}^\frac{2}{3} - 7(1-\theta)\sgn{z_1}^{1+\frac{2}{100}}+z_2 \\
    \frac{dz_2}{d\tau} &= -\frac{3}{2}\sqrt{2}L^\frac{4}{6}\theta\sgn{z_1}^\frac{1}{3}-\frac{15}{7}(1-\theta)\sgn{z_1}^{1+\frac{4}{100}}+z_3 \\
    \frac{dz_3}{d\tau} &= -\frac{11}{10}L\theta\sgn{z_1}^0-(1-\theta)\sgn{z_1}^{1+\frac{6}{100}} + \hat{\delta}(\tau)\\
\end{aligned}
\end{equation}
with $\hat{\delta}(\tau)\leq L=2.5$ and $\theta=0$ for $\tau\leq T_\theta$ and $\theta=1$ otherwise, is fixed-time stable.
Moreover, the vector field $F(x_1)$ is obtained from these parameters and Proposition \ref{Basin}. Let us apply Corollary \ref{Cor:BasinDiff} with $n_d=1$. In this case $\rho(\tau)$ is chosen as in Example \ref{ex:cruz_zavala}. Figure \ref{fig:cruz_zavala} shows the trajectories of $x(t)=(x_1(t),x_2(t),x_3(t))^T$ with $x(0) = (100,0,0)^T$ as well as the values of $y(t)$ and $\dot{y}(t)$. Moreover, the performance of the system with measurement noise $0.01\cos(10t)+0.001\cos(30t)$ is shown.

\begin{figure}
    \centering
\def\svgwidth{8.5cm}    
\begingroup%
  \makeatletter%
  \providecommand\color[2][]{%
    \errmessage{(Inkscape) Color is used for the text in Inkscape, but the package 'color.sty' is not loaded}%
    \renewcommand\color[2][]{}%
  }%
  \providecommand\transparent[1]{%
    \errmessage{(Inkscape) Transparency is used (non-zero) for the text in Inkscape, but the package 'transparent.sty' is not loaded}%
    \renewcommand\transparent[1]{}%
  }%
  \providecommand\rotatebox[2]{#2}%
  \newcommand*\fsize{\dimexpr\f@size pt\relax}%
  \newcommand*\lineheight[1]{\fontsize{\fsize}{#1\fsize}\selectfont}%
  \ifx\svgwidth\undefined%
    \setlength{\unitlength}{352.15393066bp}%
    \ifx\svgscale\undefined%
      \relax%
    \else%
      \setlength{\unitlength}{\unitlength * \real{\svgscale}}%
    \fi%
  \else%
    \setlength{\unitlength}{\svgwidth}%
  \fi%
  \global\let\svgwidth\undefined%
  \global\let\svgscale\undefined%
  \makeatother%
  \begin{picture}(1,0.80971289)%
    \lineheight{1}%
    \setlength\tabcolsep{0pt}%
    \put(0,0){\includegraphics[width=\unitlength]{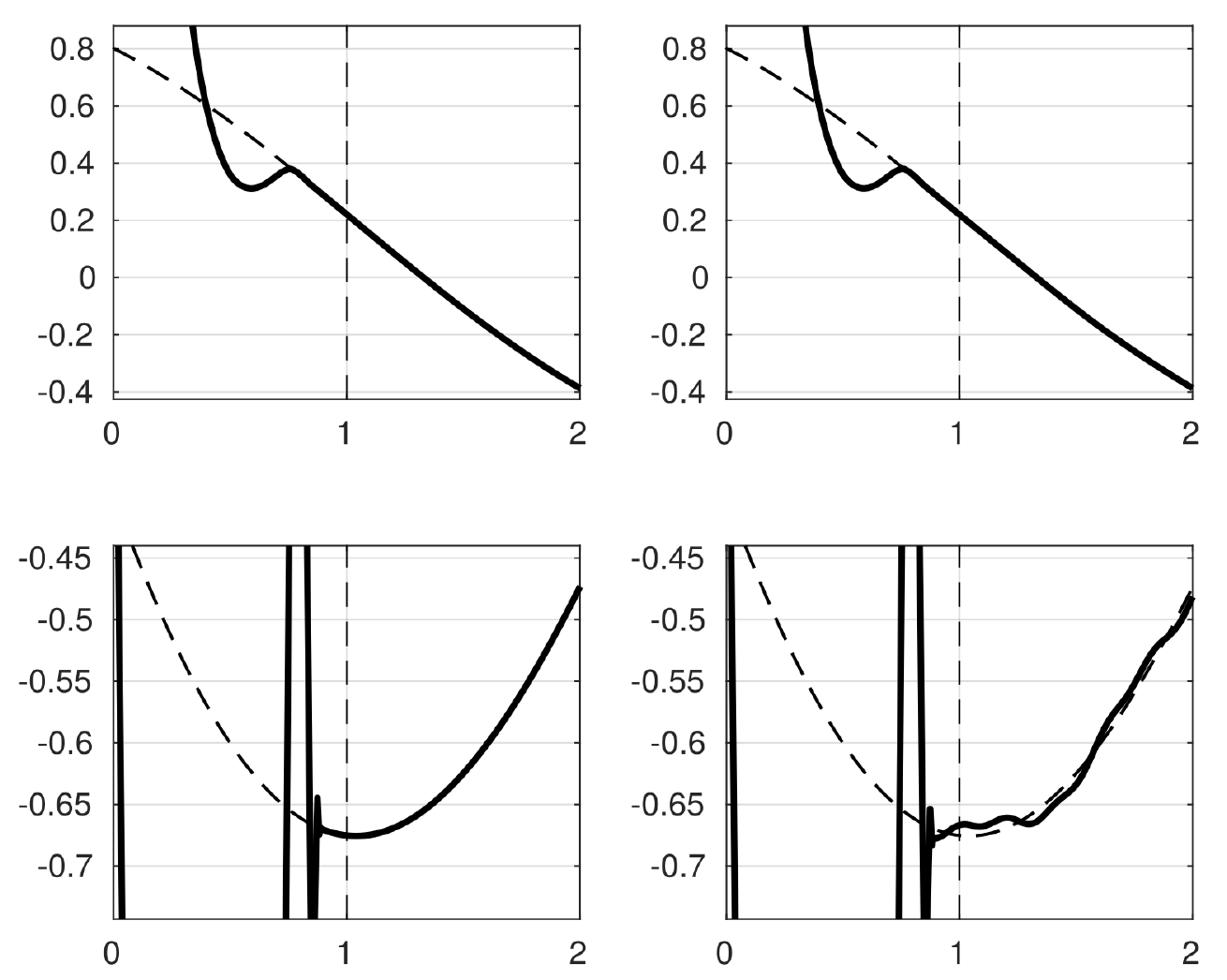}}%
        \put( 0.160482251,0.784956467){  \rotatebox{0}{\makebox(0,0)[lb]{$x_1(t),y(t)$}}
    }%
     \put( 0.670482251,0.784956467){  \rotatebox{0}{\makebox(0,0)[lb]{$x_1(t),y(t)$}}
    }%
    \put( 0.160482251,0.35956467){  \rotatebox{0}{\makebox(0,0)[lb]{$x_2(t),\dot{y}(t)$}}
    }%
     \put( 0.670482251,0.35956467){  \rotatebox{0}{\makebox(0,0)[lb]{$x_2(t),\dot{y}(t)$}}
    }%
    \put( 0.490482251,0.45956467){\rotatebox{0}{\makebox(0,0)[lb]{$t$}}}%
    \put( 0.99482251,0.45956467){\rotatebox{0}{\makebox(0,0)[lb]{$t$}}}%
    \put( 0.490482251,0.03956467){\rotatebox{0}{\makebox(0,0)[lb]{$t$}}}%
    \put( 0.99482251,0.03956467){\rotatebox{0}{\makebox(0,0)[lb]{$t$}}}%
    
  \end{picture}%
\endgroup%
    \caption{Online differentiation of the signal $y(t)=-0.4\sin(t)+0.8\cos(0.8t)$ with predefined-time convergence at $T_c=1$ of Example~\ref{ex:marco_tulio}, in the noiseless case (left) and with noise (right).}
    \label{fig:marco_tulio}
\end{figure}
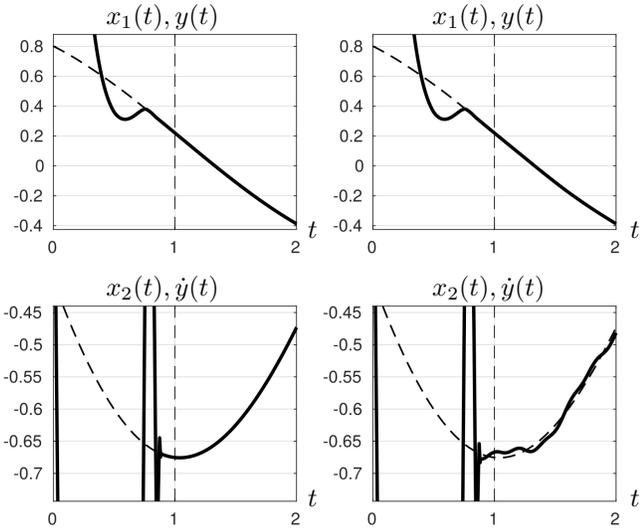
\end{example}

\begin{example}
\label{ex:cruz_zavala}
Let $g_1(x_1) = \sgn{x_1}^\frac{1}{2}+\sgn{x_1}^\frac{3}{2}$, $g_2(x_1) = \frac{1}{2}\sign{x_1}+2 x_1 + \frac{3}{2}\sgn{x_1}^2$.
According to \cite{Cruz-Zavala2011}, the system
\begin{equation}
\begin{aligned}
    \frac{dz_1}{d\tau} &=-2\sqrt{3}g_1(z_1) + z_2 \\
    \frac{dz_2}{d\tau} &= -6g_2(z_1) + \hat{\delta}(\tau)
\end{aligned}
\end{equation}
with $\hat{\delta}(\tau)\leq 2.5$ is fixed-time stable with $T_{\max}^* = 233.7349$.
Moreover, the vector field $F(x_1)$ is obtained from these parameters and Proposition \ref{Basin}. Let us apply Corollary~\ref{Cor:BasinDiff} with $n_d=1$. In this case, $t_0=0$, $\rho(\tau) = \frac{1}{T_c}\exp(\tau)$ and thus, $\kappa(t-t_0) = \frac{1}{T_c-(t-t_0)}$. The resulting system was simulated to obtain the derivative of $y(t)=-0.4\sin(t)+0.8\cos(0.8t)$. Figure~\ref{fig:cruz_zavala} shows, on the left column, the trajectories of $x(t)=(x_1(t),x_2(t))^T$ with $x(0) = (100,0)^T$ as well as the values of $y(t)$ and $\dot{y}(t)$, and, on the right column, the performance of the system with measurement noise $0.01\cos(10t)+0.001\cos(30t)$.

\begin{figure}
    \centering
\def\svgwidth{9.5cm}    
\begingroup%
  \makeatletter%
  \providecommand\color[2][]{%
    \errmessage{(Inkscape) Color is used for the text in Inkscape, but the package 'color.sty' is not loaded}%
    \renewcommand\color[2][]{}%
  }%
  \providecommand\transparent[1]{%
    \errmessage{(Inkscape) Transparency is used (non-zero) for the text in Inkscape, but the package 'transparent.sty' is not loaded}%
    \renewcommand\transparent[1]{}%
  }%
  \providecommand\rotatebox[2]{#2}%
  \newcommand*\fsize{\dimexpr\f@size pt\relax}%
  \newcommand*\lineheight[1]{\fontsize{\fsize}{#1\fsize}\selectfont}%
  \ifx\svgwidth\undefined%
    \setlength{\unitlength}{352.15393066bp}%
    \ifx\svgscale\undefined%
      \relax%
    \else%
      \setlength{\unitlength}{\unitlength * \real{\svgscale}}%
    \fi%
  \else%
    \setlength{\unitlength}{\svgwidth}%
  \fi%
  \global\let\svgwidth\undefined%
  \global\let\svgscale\undefined%
  \makeatother%
  \begin{picture}(1,0.80971289)%
    \lineheight{1}%
    \setlength\tabcolsep{0pt}%
    \put(0,0){\includegraphics[width=\unitlength]{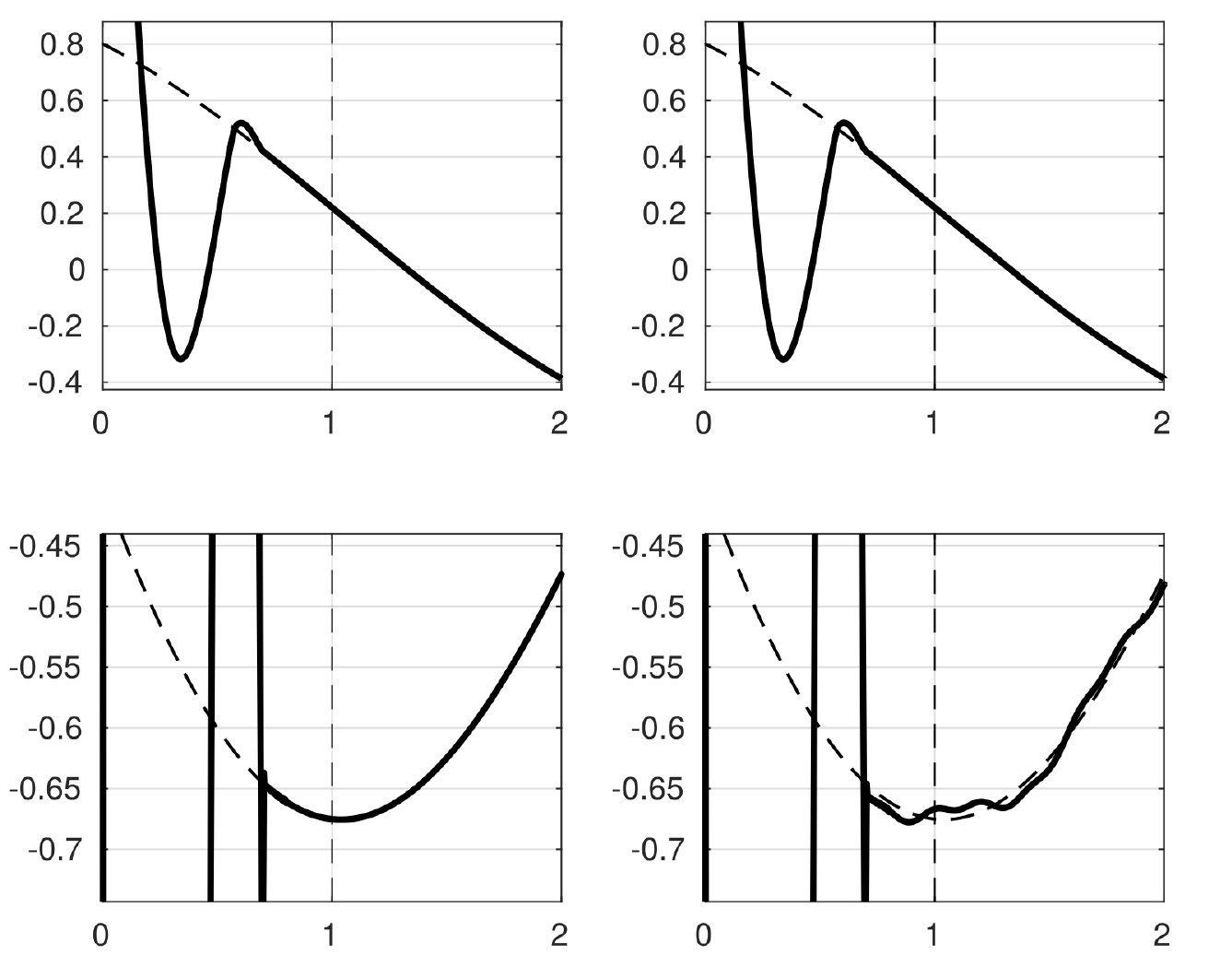}}%
        \put( 0.160482251,0.784956467){  \rotatebox{0}{\makebox(0,0)[lb]{$x_1(t),y(t)$}}
    }%
     \put( 0.670482251,0.784956467){  \rotatebox{0}{\makebox(0,0)[lb]{$x_1(t),y(t)$}}
    }%
    \put( 0.160482251,0.35956467){  \rotatebox{0}{\makebox(0,0)[lb]{$x_2(t),\dot{y}(t)$}}
    }%
     \put( 0.670482251,0.35956467){  \rotatebox{0}{\makebox(0,0)[lb]{$x_2(t),\dot{y}(t)$}}
    }%
    
    \put( 0.470482251,0.45956467){\rotatebox{0}{\makebox(0,0)[lb]{$t$}}}%
    \put( 0.95482251,0.45956467){\rotatebox{0}{\makebox(0,0)[lb]{$t$}}}%
    \put( 0.470482251,0.03956467){\rotatebox{0}{\makebox(0,0)[lb]{$t$}}}%
    \put( 0.95482251,0.03956467){\rotatebox{0}{\makebox(0,0)[lb]{$t$}}}%
  \end{picture}%
\endgroup%
    \caption{Online differentiation of the signal $y(t)=-0.4\sin(t)+0.8\cos(0.8t)$ with predefined-time convergence at $T_c=1$ of Example~\ref{ex:cruz_zavala}, in the noiseless case (left) and with noise (right).}
    \label{fig:cruz_zavala}
\end{figure}
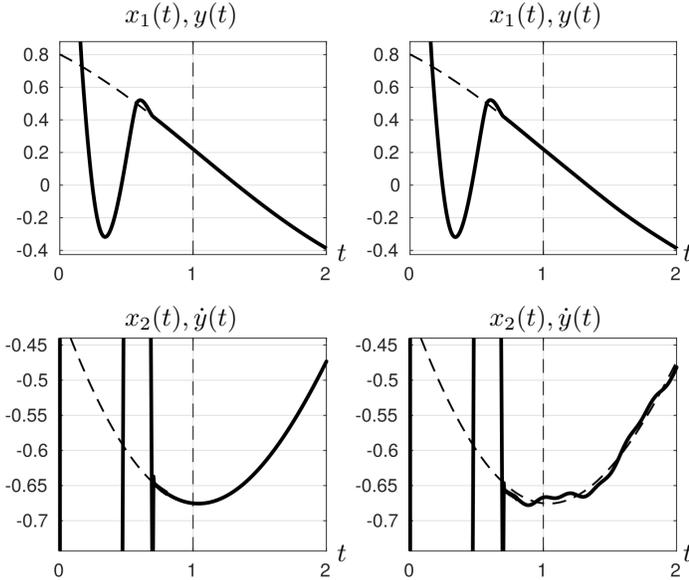

\end{example}

\section{Conclusion}
\label{Sec:Conclu}
This technical note introduced a new class of non-autonomous fixed-time stable systems with predefined \textit{UBST}, which is based on time-varying gains. This new result enables the design of unknown input observers and online differentiation algorithms where an upper bound for the convergence time is set a priori as a parameter of the algorithm. 

We present conditions such that the settling time of every nonzero trajectory is precisely the predefined one. Moreover, we provide a methodology, to derive an autonomous fixed-time stable system with predefined \textit{UBST} from an homogeneous algorithm with fixed-time convergence whose \textit{UBST} estimate is too-conservative, significantly reducing the over-estimation. Unlike existing prescribed-time observers such as~\cite{Holloway2019}, in our approach, the time-varying gain remains bounded, which makes our approach realizable. Another advantage with respect to~\cite{Holloway2019} is that predefined-time convergence is guaranteed even in the presence of disturbances, which enables the application to unknown input observers and online differentiation algorithms. 

One limitation of our approach is that, designing a non-autonomous fixed-time system guaranteeing bounded time-varying gains for every initial condition, requires that the auxiliary system~\eqref{Eq:TauSystem} (or system~\eqref{Eq:System1} when Proposition~\eqref{Basin} is used) is fixed-time stable with a known \textit{UBST}. Thus, developing autonomous fixed-time systems with an explicit \textit{UBST} is required to broaden the applicability of our approach. We hope our work will catalyze efforts towards achieving this goal. Fortunately,  when initial conditions are bounded as often happens in practice, we can exploit the results on Lyapunov analysis for the arbitrary order \textit{HOSM} differentiator presented in~\cite{Cruz2018} to produce a non-autonomous fixed-time system guaranteeing bounded gains.

\appendix
\subsection{Preliminaries on homogeneity and time-scale transformations}

\subsubsection{Homogeneity}

\begin{definition}
\label{Def.Homogeneous}
\cite{Bhat2005}
Consider the vector of weights $\mathbf{r}=[r_1,\ldots,r_n]^T$, where $r_i>0$, $i=1,\ldots,n$. A vector field $f:\mathbb{R}^n\rightarrow\mathbb{R}^n$ is said to be $\mathbf{r}$-homogeneous of degree $d\geq-\min_{1\leq i\leq n}r_i$ with respect to the dilation matrix $\Delta_{\mathbf{r}}(\lambda)=diag(r_1,\ldots,r_n)$, where $\lambda>0$, if
$$
f(x)=\lambda^{-d}\Delta_{\mathbf{r}}^{-1}(\lambda)f(\Delta_{\mathbf{r}}(\lambda)x)
$$
or equivalently if the $i$-th element, $i=1,\ldots,n$, of $f(x)$ satisfies
$$f_i(x)=\lambda^{-(d+r_i)}f_i(\Delta_{\mathbf{r}}(\lambda)x).$$

\end{definition}



\begin{theorem}
\label{Th:existenceLyapunov}
Let the vector field $f:\mathbb{R}^n\rightarrow\mathbb{R}^n$ be discontinuous and $\mathbf{r}$-homogeneous with negative degree $d\in\mathbb{R}$, with respect to the dilation $\Delta_{\mathbf{r}}(\lambda)$. If for the differential inclusion  $\dot{x}=-f(x)$ the origin is (strongly) globally asymptotically stable then for all $k>-d$, there exists a function $V:\mathbb{R}^n\to \mathbb{R}$ which satisfies:
\begin{enumerate}
    \item $V$ is $\mathbf{r}$-homogeneous of degree $k\in\mathbb{R}$, i.e. $V(\Delta_{\mathbf{r}}(\lambda)x) = \lambda^kV(x)$ for all $\lambda>0$.
    \item There exists positive constants $c_3$ and $p<1$ such that $\frac{\partial V}{\partial x}f(x) < -c_3V^p, \ \ \forall x\neq 0$.    
    \item $V(0)=0$, $V(x)>0 \ \forall x\neq 0$ and $V(x)\to+\infty$ as $\|x\|\to +\infty$
    \item There exist two positive constants $c_1,c_2$ such that $c_1\|x\|_{\mathbf{r}}^k\leq V(x) \leq c_2\|x\|_{\mathbf{r}}^k$
\end{enumerate}
\end{theorem}
\begin{proof}
The first item is found in~\cite[Theorem~4.1]{Bernuau2013b}. Items $(2)$--$(4)$ follow from the homogeneity property of $V$ and the negative homogeneity degree of $f(x)$~\cite{Bhat2005}.
\end{proof}




\subsubsection{Time-scale transformations}

As in~\cite{Pico2013,aldana2019design}, the trajectories corresponding to the system solutions of~\eqref{eq:sys} are interpreted, in the sense of differential geometry~\cite{Kuhnel2015}, as regular parametrized curves. Since we apply regular parameter transformations over the time variable, then without ambiguity, this reparametrization is sometimes referred to as time-scale transformation.

\begin{definition}\cite[Definition~2.1]{Kuhnel2015}
\label{Def:RegularParamCurve}
A regular parametrized curve, with parameter $t$, is a $C^1(\mathcal{I})$ immersion $c: \mathcal{I}\to \mathbb{R}$, defined on a real interval $\mathcal{I} \subseteq \mathbb{R}$. This means that $\frac{dc}{dt}\neq 0$ holds everywhere.
\end{definition}

\begin{definition}\cite[Pg.~8]{Kuhnel2015}
\label{Def:RegularCurve}
A regular curve is an equivalence class of regular parametrized curves, where the equivalence relation is given by regular (orientation preserving) parameter transformations $\varphi$, where $\varphi:~\mathcal{I}~\to~\mathcal{I}'$ is $C^1(\mathcal{I})$, bijective and $\frac{d\varphi}{dt}>0$. Therefore, if $c:\mathcal{I}\to\mathbb{R}$ is a regular parametrized curve and $\varphi:\mathcal{I}\to \mathcal{I}'$ is a regular parameter transformation, then $c$  and  $c\circ\varphi:\mathcal{I}'\to\mathbb{R}$ are considered to be equivalent.
\end{definition}



\begin{lemma}
\label{Lemma:ParamTransf}
\cite{aldana2019design}
Let the map $\psi:[0,\infty)\to[0,T_c)$ be given by~\eqref{Eq:NonAutpsi},
where $T_c>0$ is a parameter and $\Phi(\cdot)$ satisfies Assumption~\ref{Assump:NonAut}. Moreover, let the system $\frac{dy}{d\tau} = g(y)\in\mathbb{R}^n$ with $y(0) = y_0$, be such that the origin is globally stable with settling time $\mathcal{T}(y_0)$ and has a unique solution. 
Then, the origin of the system
\begin{equation}
    \dot{x} =  \frac{1}{T_c}\Phi(\psi^{-1}(t-t_0))^{-1}g(x), \ \ x(0) = x_0
\end{equation}
is fixed time stable with settling time function given by $T(x_0,t_0)=\lim_{\tau\to\mathcal{T}(x_0)}\psi(\tau)\leq T_c$ and has a unique solution $\forall t\in[t_0,\infty)$ for each initial condition $x_0$. Furthermore, the bijective function $\varphi:\mathcal{I}=[t_0,t_0 + \lim_{\tau\to\mathcal{T}(x_0)}\psi(\tau))\to [0,\mathcal{T}(x_0))$  defined by $\varphi^{-1}(\tau) = \psi(\tau) + t_0$  is a parameter transformation.
\end{lemma}

\subsection{Proof of the main theorem}
\label{AppendProofMain}
\begin{proof}[Proof of Theorem~\ref{Th:MainResult}]
Let us consider the change of coordinates $z_i=(r\kappa(t-t_0))^{1-i}y_i$ ($i=1,\cdots,n$). Hence, in the new coordinate, the dynamic of the system is given by
\begin{multline}
\label{Eq:zi}
    \dot{z}_i=\kappa(t-t_0)\left(-(i-1)\kappa(t-t_0)^{-2}\dot{\kappa}(t-t_0)z_i \right.\\ \left.+ rk_if_i(z_1) + rz_{i+1}\right) 
\end{multline}
for $i=1,\cdots,n-1$, and
\begin{multline}
\label{Eq:zn}
    \dot{z}_n=\kappa(t-t_0)\left(-(n-1)\kappa(t-t_0)^{-2}\dot{\kappa}(t-t_0)z_n\right.\\ \left.+ rk_nf_n(z_1)+r\kappa(t-t_0)^{-n}\delta(t)\right) 
\end{multline}

Now, consider the bijective map $\varphi:[t_0,t_0 + \lim_{\tau\to\mathcal{T}(z_0)}\psi(\tau))\to \mathcal{I}'=[0,\mathcal{T}(z_0))$  defined by $\varphi^{-1}(\tau) = \psi(\tau) + t_0$, where
\begin{equation}
\psi(\tau)=\int_{0}^\tau\rho^{-1}(\xi)d\xi,  \ \ \tau\in \mathcal{I}',
\end{equation}
which, according to Lemma~\ref{Lemma:ParamTransf}, defines the parameter transformation $t=\varphi^{-1}(\tau)$. Notice that $t-t_0=\psi(\tau)$, one can get $$\frac{dt}{d\tau}=\frac{d(t-t_0)}{d\tau}=\rho^{-1}(\tau)=\rho^{-1}(\psi^{-1}(t-t_0))=\kappa^{-1}(t-t_0).$$
Hence, by the chain rule $\frac{dz}{d\tau}=\left. \frac{dz}{dt}\frac{dt}{d\tau}\right\rvert_{t=\psi(\tau)+t_0}$, one can derive
\begin{align}
\dot{\kappa}(\hat{t}) &= \frac{d}{d\hat{t}}\rho(\psi^{-1}(\hat{t}))=\rho'(\psi^{-1}(\hat{t}))\frac{d\psi^{-1}(\hat{t})}{d\hat{t}}\\
&= \rho'(\psi^{-1}(\hat{t}))\left[\psi'(\psi^{-1}(\hat{t}))\right]^{-1} \\
&= \rho'(\psi^{-1}(\hat{t}))\rho(\psi^{-1}(\hat{t})),
\end{align}
where $\hat{t} = t-t_0$. Hence, $\dot{\kappa}(\psi(\tau)) = \frac{d\rho(\tau)}{d\tau}\rho(\tau)$.
Thus, under such parameter transformation, the dynamic of~\eqref{Eq:zi}-\eqref{Eq:zn} can be written as
\begin{align}
    \frac{dz_i}{d\tau} &= -(i-1)\rho(\tau)^{-1}\frac{d\rho(\tau)}{d\tau}z_i+ rk_if_i(z_1) + rz_{i+1}
\intertext{for $i=1,\cdots,n-1$, and }
    \frac{dz_n}{d\tau}&=-(n-1)\rho(\tau)^{-1}\frac{d\rho(\tau)}{d\tau}z_n+ rk_nf_n(z_1)+(r\rho(\tau))^{-n}\hat{\delta}(\tau)
\end{align}
where $\hat{\delta}(\tau)=\delta(\psi(\tau)+t_0)$, which, can be written as~\eqref{Eq:TauSystem}.

Since system~\eqref{Eq:TauSystem} is asymptotically stable and has a settling time function $\mathcal{T}(z_0)$, then, due to Lemma~\ref{Lemma:ParamTransf} the settling time of~\eqref{Eq:PrescDiff} is given by $T(y_0,t_0)=\lim_{\tau\to\mathcal{T}(z_0)}\psi(\tau)=T_c\int_{0}^{\mathcal{T}(z_0)}\Phi(\xi)d\xi\leq T_c \int_{0}^{T_f}\Phi(\xi)d\xi \leq \eta T_c \leq T_c$. Thus, the settling time function of ~\eqref{Eq:PrescDiff} is upper bounded by $T_c$.


\end{proof}

\subsection{Proof of the Propositions}
\label{AppendixPropProofs}

\begin{proof}[Proof of Proposition~\ref{Prop:Main}]
Clearly, if $\mathcal{T}(z_0)=+\infty$, for all $z_0\in\mathbb{R}^n\setminus \{0\}$, then, $T(y_0,t_0)=\lim_{\tau\to\mathcal{T}(z_0)}\psi(\tau)=T_c\int_{0}^{\mathcal{T}(z_0)}\Phi(\xi)d\xi=T_c$, for all $y_0\in\mathbb{R}^n\setminus \{0\}$. Thus, \textit{item 1)} holds. Moreover, if $\mathcal{T}(z_0)$ is finite but radially unbounded, i.e. $\lim_{\|z_0\|\to\infty}\mathcal{T}(z_0)=\infty$ then
$T(y_0,t_0)=T_c\int_{0}^{\mathcal{T}(z_0)}\Phi(\xi)d\xi<T_c$. Therefore, $\sup_{(y_0,t_0) \in \mathbb{R}^n\times\mathbb{R}_+} T(y_0,t_0)= T_c\lim_{\|z_0\|\to\infty}\int_{0}^{\mathcal{T}(z_0)}\Phi(\xi)d\xi=T_c$. Thus, $T_c$ is the least \textit{UBST}. Hence, \textit{item 2)} holds. Finally, if the origin of~\eqref{Eq:PrescDiff} is fixed-time stable and $T_f$ is the \textit{UBST}, i.e. $\sup_{z_0 \in \mathbb{R}^n}\mathcal{T}(z_0)<T_f$, then there exists $\hat{T}_c<T_c$ such that $\sup_{(y_0,t_0) \in \mathbb{R}^n\times\mathbb{R}_+} T(y_0,t_0)< T_c\int_{0}^{T_f}\Phi(\xi)d\xi=\hat{T}_c<T_c$. Therefore, \textit{item 3)} holds.

\end{proof}

\begin{proof}[Proof of Proposition~\ref{Prop:Lineal}]
Note that the auxiliary system \eqref{Eq:TauSystem} can be expressed as 
$$
\frac{dz}{d\tau} = rAz - \rho^{-1}\frac{d\rho}{d\tau}Mz + (r\rho)^{-n}D\hat{\delta}.
$$
Let $V = z^TPz$ be a candidate Lyapunov function for \eqref{Eq:TauSystem}. Thus,
\begin{align}
&\frac{dV}{d\tau} = rz^T(A^TP+PA)z - 2\rho^{-1}\frac{d\rho}{d\tau}z^TPMz + 2(r\rho)^{-n}z^TPD\hat{\delta}\\&
\leq  -r\|z\|^2 + 2\rho^{-1}\frac{d\rho}{d\tau}\|Pz\|\|Mz\| + 2(r\rho)^{-n}\|Pz\||\hat{\delta}|\\&
\leq  -\|z\|\left(\left(r - 2\rho^{-1}\frac{d\rho}{d\tau}\lambda_{\max}(P)(n-1)\right)\|z\|-2L\frac{\lambda_{\max}(P)}{r^n\rho^{n}}\right)
\end{align}
since $\|Pz\|\leq\lambda_{\max}(P)\|z\|$ and $\|Mz\|\leq(n-1)\|z\|$. Note that if condition 1) is fulfilled, then the term $r - 2\rho^{-1}\frac{d\rho}{d\tau}\lambda_{\max}(P)(n-1) > 0$ for all $\tau\geq0$. Moreover, if condition 2) is satisfied, then there exist a positive constant $\tau^*$ at which $r - 2\rho^{-1}\frac{d\rho}{d\tau}\lambda_{\max}(P)(n-1) > 0$ for all $\tau\geq \tau^*$ and $r> 0$. Hence, $\frac{dV}{d\tau}\leq 0$ for $\tau\geq\tau^*$ and  $z$ outside the region  $\{z\in\mathbb{R}^n:\|z\|\leq\alpha_1(\tau)\}$ where 
$$
\alpha_1(\tau) = \frac{2L\lambda_{\max}(P)}{r^n\rho^n(r - 2\rho^{-1}\frac{d\rho}{d\tau}\lambda_{\max}(P)(n-1))}. 
$$
Note that $\lim_{\tau\to\infty}\alpha_1(\tau)=0$ and therefore \eqref{Eq:TauSystem} is asymptotically stable. 
Thus, the result follows from Theorem~\ref{Th:MainResult} \textit{item 1)}.
\end{proof}

\begin{proof}[Proof of Proposition~\ref{Prop:STFixed}]
Let the vector field $\mathcal{F}: \mathbb{R}^n \rightarrow \mathbb{R}^n$ defined as $\mathcal{F}(z) = F(z_1)+A_0z$. Consider the vector of weights $\mathbf{r}=[n,n-1,\ldots,1]^T$ and note that $\mathcal{F}(z)$ is $\mathbf{r}$-homogeneous of degree $d = -1$. Hence, there exists a function $V:\mathbb{R}^n\to\mathbb{R}$ satisfying all conditions of Theorem~\ref{Th:existenceLyapunov}. Let $S_r = \{z\in \mathbb{R}^n:\|z\|_\mathbf{r} = 1\}$ and note that the map $\Delta:(0,+\infty)\times S_r\to\mathbb{R}^n\setminus\{0\}$  defined by $\Delta(\lambda,y) = \Delta_\mathbf{r}(\lambda)y$ is surjective, where $\Delta_\mathbf{r}(\lambda)$ is given in Definition \ref{Def.Homogeneous}. Hence, there exists a value of $\lambda>0$ which maps $y\in S_r$ to any $z = \Delta_\mathbf{r}(\lambda)y\in\mathbb{R}^n\setminus\{0\}$. Note that $\frac{dz_i}{dy_i} = \lambda^{r_i}$ such that $\frac{\partial V}{\partial y_i} = \frac{\partial V}{\partial z_i}\lambda^{r_i}$ and $\frac{\partial V}{\partial y} = \frac{\partial V}{\partial z}\Delta_\mathbf{r}(\lambda)$. Therefore, the evolution of $V(z)$ along system \eqref{Eq:TauSystem} can be expressed in terms of $y$ as
\begin{align}
\frac{dV}{d\tau}=& \frac{\partial V}{\partial z}\left(r\mathcal{F}(z) -\rho(\tau)^{-1}\frac{d\rho(\tau)}{d\tau}Mz+(r\rho(\tau))^{-n}D\hat{\delta}(\tau)\right)\\
=&\frac{\partial V}{\partial y}\Delta_\mathbf{r}(\lambda)^{-1}\left(r\mathcal{F} (\Delta_\mathbf{r}(\lambda)y)-\rho(\tau)^{-1}\frac{d\rho(\tau)}{d\tau}M\Delta_\mathbf{r}(\lambda)y\right.\\
 &\left.+(r\rho(\tau))^{-n}D\hat{\delta}(\tau)\right)\\
=&\frac{\partial V}{\partial y}\Delta_\mathbf{r}(\lambda)^{-1}\left(r\lambda^{d}\Delta_\mathbf{r}(\lambda)\mathcal{F} (y)-\Delta_\mathbf{r}(\lambda)\rho(\tau)^{-1}\frac{d\rho(\tau)}{d\tau}My\right.\\
 &\left.+(r\rho(\tau))^{-n}D\hat{\delta}(\tau)\right)\\
=&\frac{\partial V}{\partial y}\left(r\lambda^{d}\mathcal{F} (y)-\rho(\tau)^{-1}\frac{d\rho(\tau)}{d\tau}My\right.\\
&\left.+\Delta_\mathbf{r}(\lambda)^{-1}(r\rho(\tau))^{-n}D\hat{\delta}(\tau)\right)\\
<&-c_3r\lambda^dV(y)^p + \rho(\tau)^{-1}\frac{d\rho(\tau)}{d\tau}\left\|\frac{\partial V}{\partial y}\right\|\|My\|\\
 &+(r\rho(\tau))^{-n}\left\|\frac{\partial V}{\partial y}\right\|\lambda^{-r_n}L\\
<&-c_3c_1r\lambda^d\|y\|_\mathbf{r}^{kp}+\rho(\tau)^{-1}\frac{d\rho(\tau)}{d\tau}b_1\\
 &+ (r\rho(\tau))^{-n}b_2
\end{align}
with $b_1 = \sup_{y\in S_r}\left\|\frac{\partial V}{\partial y}\right\|\|My\|$ and $b_2=\sup_{y\in S_r}\left\|\frac{\partial V}{\partial y}\right\|\lambda^{-r_n}L$. Note that $\|y\|_\mathbf{r}^{kp}=1$ and $\|z\| = \|\Delta_\mathbf{r}(\lambda)y\| > \lambda^{\tilde{r}}\|y\|\geq \lambda^{\tilde{r}}R$ where $\tilde{r} = \text{argmin}_{r_i} \lambda^{r_i}$ and $R = \min_{y\in S_r}\|y\|$. Hence, $\lambda > \|z\|^{1/\tilde{r}}R^{-1/\tilde{r}}$ and
\begin{align}
    \frac{dV}{d\tau}&<-c_3c_1rR^{-d/\tilde{r}}\|z\|^{d/\tilde{r}}+\rho(\tau)^{-1}\frac{d\rho(\tau)}{d\tau}b_1 + (r\rho(\tau))^{-n}b_2
\end{align}
Therefore, $\frac{dV}{d\tau} < 0$ for $z$ outside the region  $\{z\in\mathbb{R}^n:\|z\|\leq\alpha_2(\tau)\}$ where 
$$
\alpha_2(\tau) = \left(
-c_3c_1rR^{-d/\tilde{r}}\right)^{-\frac{\tilde{r}}{d}}\left(\rho(\tau)^{-1}\frac{d\rho(\tau)}{d\tau}b_1 + (r\rho(\tau))^{-n}b_2\right)^{\frac{\tilde{r}}{d}}
$$
Note that $\lim_{\tau\to\infty}\alpha_2(\tau)=0$ and therefore \eqref{Eq:TauSystem} is asymptotically stable. Let the change of coordinates $w_1 = z_1$, $w_2=z_2$ and $w_i=z_i-\sum_{j=1}^{i-2}\frac{d^{j-1}}{d\tau^{j-1}}\left(\rho(\tau)^{-1}\frac{d\rho(\tau)}{d\tau}z_{i-j}\right)(i-j-1)$ for $3\leq i\leq n$. Then,
\begin{align*}
\frac{dw_i}{d\tau}=&\frac{dz_i}{d\tau}-\sum_{j=1}^{i-2}\frac{d^{j}}{d\tau^{j}}\left(\rho(\tau)^{-1}\frac{d\rho(\tau)}{d\tau}z_{i-j}\right)(i-j-1)\\
=&\frac{dz_i}{d\tau} + (i-1)\rho(\tau)\frac{d\rho(\tau)}{d\tau}z_i \\
&- \sum_{j=0}^{i-2}\frac{d^{j}}{d\tau^{j}}\left(\rho(\tau)^{-1}\frac{d\rho(\tau)}{d\tau}z_{i-j}\right)(i-j-1)\\
=&\frac{dz_i}{d\tau} + (i-1)\rho(\tau)\frac{d\rho(\tau)}{d\tau}z_i \\
&- \sum_{j=1}^{i-1}\frac{d^{j-1}}{d\tau^{j-1}}\left(\rho(\tau)^{-1}\frac{d\rho(\tau)}{d\tau}z_{i-j+1}\right)(i-j)\\
=&\frac{dz_i}{d\tau} + (i-1)\rho(\tau)\frac{d\rho(\tau)}{d\tau}z_i + (w_{i+1}-z_{i+1})\\
=&k_i\sgn{w_1}^{\frac{n-i}{n}} + w_{i+1}
\end{align*}
for $i\neq n$ and $\frac{dw_n}{d\tau} = k_n\sign{w_1} + Q(w,\tau)$, where $Q(w,\tau) = \sum_{j=1}^{n-2}\frac{d^{j-1}}{d\tau^{j-1}}\left(\rho(\tau)^{-1}\frac{d\rho(\tau)}{d\tau}z_{n-j}\right)(n-j-1) + (\rho(\tau))^{-n}\hat{\delta}(\tau)$. This can be written in a compact form as $\frac{dw}{d\tau} = G(w) + DQ(z,t)$. Note that since $\lim_{\tau\to+\infty}\rho(\tau)^{-n}= 0$, and that $z$ converges to the origin asymptotically, there exists a constant $\tau^*$ such that $|Q(w,\tau)|\leq L,\ \forall \tau\geq \tau^*$. Therefore, by \cite{Levant2001}, the origin of $\frac{dw}{d\tau} = G(w) + DQ(z,t)$ is finite time stable. Thus, the origin of \eqref{Eq:TauSystem} is finite time stable. Thus, the result follows from Theorem~\ref{Th:MainResult} item 2).
\end{proof}

\begin{proof}[Proof of Proposition~\ref{Basin}]
Notice that~\eqref{Eq:TauSystem} becomes 
\begin{equation}
    \frac{dz}{d\tau}=F(z_1)+A_0z-\alpha Mz+D\hat{\delta}(\tau).
\end{equation}
Consider the coordinate change $z=\mathcal{Q}\hat{z}$, thus the dynamic of $\hat{z}$ is given by
$
    \frac{d\hat{z}}{d\tau}=\mathcal{Q}^{-1}F(Q\hat{z}_1)+\mathcal{Q}^{-1}(A_0-\alpha M)\mathcal{Q}\hat{z}+\mathcal{Q}^{-1}D\hat{\delta}(\tau)
$,
i.e.
\begin{align}
\frac{d\hat{z}_1}{d\tau}&=k_1g_1(\hat{z}_1)+\hat{z}_2\\
\intertext{for $i=1,\ldots,n-1$, and}
\frac{d\hat{z}_n}{d\tau}&=k_ng_n(\hat{z}_1)+\hat{\delta}(\tau)
\end{align}
Thus, system~\eqref{Eq:TauSystem} is asymptotically stable with settling time function $\mathcal{T}(z_0)$ satisfying $\mathcal{T}(z_0)\leq T_{f}$ for all $z_0\in\mathbb{R}^n$. 
Moreover, according to Lemma~\ref{Lemma:ParamTransf}, an \textit{UBST} of~\eqref{Eq:PrescDiff} is given by $\lim_{\tau\to T_{\max}^*}\psi(\tau)=\eta T_c$. It follows from Corollary~\ref{Cor:BoundedGain} that, if there is a known $T_{\max}^*<+\infty$, such that $\mathcal{T}(z_0)\leq T_{\max}^*$, then for all $t\in[t_0,t_0+\eta T_c]$, $\kappa(t-t_0)$ is bounded. 
\end{proof}

\subsection{Auxiliary Results}
\label{AppendixAuxResults}

\begin{lemma}
\label{lemma:exp_stability}
If the origin of system~\eqref{Eq:TauSystem} is exponentially stable with decay rate $c$ and if there exists a constant $\tau^*>0$ which verifies
\begin{equation}
    c>(n-1)\frac{\log(\rho(\tau))}{\tau}, \qquad  \forall \tau\geq\tau^*,
    \label{eq:exp_stability}
\end{equation}
then the solution of~\eqref{Eq:TauSystem} satisfies~\eqref{Eq:CondTrans}. 
\end{lemma}
\begin{proof}
Exponential stability of ~\eqref{Eq:TauSystem} implies that $\|z\|\leq k\|z_0\|\exp(-c\tau)$ for some $k>0$. Moreover, $\|z\|\rho(\tau)^{n-1}\leq k\|z_0\|\exp(-c\tau)\rho(\tau)^{n-1} = k\|z_0\|\exp(-c\tau + (n-1)\log(\rho(\tau)))$. If \eqref{eq:exp_stability} is verified, then $\exp(-c\tau + (n-1)\log(\rho(\tau))$ decreases for $\tau\geq\tau^*$ and approaches the origin as $\tau\to+\infty$. Therefore, $|z_i\rho(\tau)^{i-1}|\leq|z_i\rho(\tau)^{n-1}|\leq\|z\|\rho(\tau)^{n-1}\to 0$ as $\tau\to+\infty$.
\end{proof}


\end{document}